\documentclass[11pt]{amsart}
\usepackage{amsmath}
\usepackage{amsthm}
\usepackage{amssymb}
\usepackage{verbatim}
\usepackage{hyperref}
\usepackage{color}
\usepackage{cite}

\begin{document}

\newtheorem{theorem}{Theorem}    
\newtheorem{proposition}[theorem]{Proposition}
\newtheorem{conjecture}[theorem]{Conjecture}
\def\theconjecture{\unskip}
\newtheorem{corollary}[theorem]{Corollary}
\newtheorem{lemma}[theorem]{Lemma}
\newtheorem{sublemma}[theorem]{Sublemma}
\newtheorem{observation}[theorem]{Observation}
\newtheorem{remark}[theorem]{Remark}
\newtheorem{definition}[theorem]{Definition}
\theoremstyle{definition}
\newtheorem{notation}[theorem]{Notation}
\newtheorem{question}[theorem]{Question}
\newtheorem{questions}[theorem]{Questions}
\newtheorem{example}[theorem]{Example}
\newtheorem{problem}[theorem]{Problem}
\newtheorem{exercise}[theorem]{Exercise}

\numberwithin{theorem}{section} \numberwithin{theorem}{section}
\numberwithin{equation}{section}

\def\earrow{{\mathbf e}}
\def\rarrow{{\mathbf r}}
\def\uarrow{{\mathbf u}}
\def\varrow{{\mathbf V}}
\def\tpar{T_{\rm par}}
\def\apar{A_{\rm par}}

\def\reals{{\mathbb R}}
\def\torus{{\mathbb T}}
\def\heis{{\mathbb H}}
\def\integers{{\mathbb Z}}
\def\naturals{{\mathbb N}}
\def\complex{{\mathbb C}\/}
\def\distance{\operatorname{distance}\,}
\def\support{\operatorname{support}\,}
\def\dist{\operatorname{dist}\,}
\def\Span{\operatorname{span}\,}
\def\degree{\operatorname{degree}\,}
\def\kernel{\operatorname{kernel}\,}
\def\dim{\operatorname{dim}\,}
\def\codim{\operatorname{codim}}
\def\trace{\operatorname{trace\,}}
\def\Span{\operatorname{span}\,}
\def\dimension{\operatorname{dimension}\,}
\def\codimension{\operatorname{codimension}\,}
\def\nullspace{\scriptk}
\def\kernel{\operatorname{Ker}}
\def\ZZ{ {\mathbb Z} }
\def\p{\partial}
\def\rp{{ ^{-1} }}
\def\Re{\operatorname{Re\,} }
\def\Im{\operatorname{Im\,} }
\def\ov{\overline}
\def\eps{\varepsilon}
\def\lt{L^2}
\def\diver{\operatorname{div}}
\def\curl{\operatorname{curl}}
\def\etta{\eta}
\newcommand{\norm}[1]{ \|  #1 \|}
\def\expect{\mathbb E}
\def\bull{$\bullet$\ }
\def\C{\mathbb{C}}
\def\R{\mathbb{R}}
\def\Rn{{\mathbb{R}^n}}
\def\Sn{{{S}^{n-1}}}
\def\M{\mathbb{M}}
\def\N{\mathbb{N}}
\def\Q{{\mathbb{Q}}}
\def\Z{\mathbb{Z}}
\def\F{\mathcal{F}}
\def\L{\mathcal{L}}
\def\S{\mathcal{S}}
\def\supp{\operatorname{supp}}
\def\dist{\operatorname{dist}}
\def\essi{\operatornamewithlimits{ess\,inf}}
\def\esss{\operatornamewithlimits{ess\,sup}}
\def\xone{x_1}
\def\xtwo{x_2}
\def\xq{x_2+x_1^2}
\newcommand{\abr}[1]{ \langle  #1 \rangle}

\newcommand{\Norm}[1]{ \|  #1 \| }
\newcommand{\set}[1]{ \{ #1 \} }
\def\one{\mathbf 1}
\def\whole{\mathbf V}
\newcommand{\modulo}[2]{[#1]_{#2}}
\renewcommand{\thefootnote}{\fnsymbol{footnote}}
\def\scriptf{{\mathcal F}}
\def\scriptg{{\mathcal G}}
\def\scriptm{{\mathcal M}}
\def\scriptb{{\mathcal B}}
\def\scriptc{{\mathcal C}}
\def\scriptt{{\mathcal T}}
\def\scripti{{\mathcal I}}
\def\scripte{{\mathcal E}}
\def\scriptv{{\mathcal V}}
\def\scriptw{{\mathcal W}}
\def\scriptu{{\mathcal U}}
\def\scriptS{{\mathcal S}}
\def\scripta{{\mathcal A}}
\def\scriptr{{\mathcal R}}
\def\scripto{{\mathcal O}}
\def\scripth{{\mathcal H}}
\def\scriptd{{\mathcal D}}
\def\scriptl{{\mathcal L}}
\def\scriptn{{\mathcal N}}
\def\scriptp{{\mathcal P}}
\def\scriptk{{\mathcal K}}
\def\frakv{{\mathfrak V}}

\allowdisplaybreaks

\arraycolsep=1pt
%
\newtheorem*{remark0}{\indent\sc Remark}
%
\renewcommand{\proofname}{Proof.} 

\title[Characterizations of a class of Musielak--Orlicz $\rm{BMO}$ spaces ]
{Characterizations of a class of Musielak--Orlicz  $\rm{BMO}$ spaces via commutators of  Riesz potential operators}
\author[Yanyan Han]{Yanyan Han*}
\address{Yanyan Han: School of Information Network Security\\People's Public Security
         University of China\\ Beijing 100038\\ People's Republic of China}
\email{hanyanyan$\_$bj@163.com}

\author[Hongwei Huang]{Hongwei Huang}
\address{Hongwei Huang: School of Mathematical Sciences\\ Xiamen University\\
	Xiamen 361005\\	People's Republic of China}
\email{hwhuang@xmu.edu.cn}

\author[Jinghan Shao]{Jinghan Shao}
\address{Jinghan Shao: School of Mathematical Sciences\\ Xiamen University\\
	Xiamen 361005\\	People's Republic of China}
\email{18103377159@163.com}

\author[Huoxiong Wu]{Huoxiong Wu}
\address{Huoxiong Wu: School of Mathematical Sciences\\ Xiamen University\\
	Xiamen 361005\\	People's Republic of China}
\email{huoxwu@xmu.edu.cn}

\keywords{fractional integral operators, commutators, BMO spaces,  Musielak--Orlicz Hardy spaces.\\
\thanks{$^*$Corresponding author.}
}


\begin{abstract}
The fractional integral operators $I_\alpha$ can be used to characterize the  Musielak--Orlicz Hardy spaces.
This paper shows that for $b\in \rm BMO(\mathbb R^n)$, the commutators $[b,I_\alpha]$ generated by fractional integral operators $I_\alpha$ with $b$ are bounded from the Musielak--Orlicz Hardy spaces $H^{\varphi_1}(\mathbb R^n)$ to the Musielak--Orlicz spaces $L^{\varphi_2}(\mathbb R^n)$ (where $1<u<\infty$ and $\varphi_1$, $\varphi_2$ are growth functions) if and only if $b\in \mathcal {BMO}_{\varphi_1,u}(\mathbb R^n)$, which are a class of non-trivial
subspaces of $\rm BMO(\mathbb R^n)$. Additionally, we obtain the boundedness of the commutator $[b,I_\alpha]$ from  $H^{\varphi_1}(\mathbb R^n)$ to $H^{\varphi_2}(\mathbb R^n)$. The corresponding results are also provided for commutators of fractional integrals associated with general homogeneous kernels.
\end{abstract}

\maketitle

\section{Introduction}

 Let $0<\alpha<n$, and $I_\alpha$ be the Riesz potential operators in $\mathbb{R}^n$, that is,
$$I_\alpha f(x)=\int_{\mathbb{R}^n}\frac{f(y)}{|x-y|^{n-\alpha}}dy.$$
Given $b\in L_{\rm loc}(\mathbb{R}^n)$, the commutator $[b, I_\alpha]$  is defined as
$$[b, I_\alpha]f(x):=b(x)I_\alpha f(x)-I_\alpha(bf)(x)=\int_{\mathbb{R}^n}[b(x)-b(y)]\frac{f(y)}{|x-y|^{n-\alpha}}dy.$$

A well-established result in the literature is that Chanillo \cite{C} initially demonstrated the boundedness of the commutator $[b, I_\alpha]$ from $L^p(\mathbb{R}^n)$ into $L^q(\mathbb{R}^n)$ for $1<p<n/\alpha$ and $1/q=1/p-\alpha/n$, if and only if $b\in \mathrm{BMO}(\mathbb{R}^n)$. Subsequent investigations by Ding et al. \cite{DLZ1} and Cruz-Uribe et al. \cite{CF} derived weak $L\log L$ type estimate of $[b,I_\alpha]$ when $b\in \mathrm{BMO}(\mathbb{R}^n)$. Further advancements were made in \cite{ST1}, where it was established that $[b, I_\alpha]$ maps $L_{\omega^p}^p(\mathbb{R}^n)$ to $L_{\lambda^q}^q(\mathbb{R}^n)$ provided $b\in \rm BMO_{\omega/\lambda}(\mathbb{R}^n)$, $1<p<n/\alpha$, $1/q=1/p-\alpha/n$, and $\omega,\,\lambda\in A_{p,q}$. Holmes et al. \cite{HRS} recently reinforced this conclusion by proving the equivalence that the commutator $[b, I_\alpha]$ is bounded from $L_{\omega^p}^p(\mathbb{R}^n)$ to $L_{\lambda^q}^q(\mathbb{R}^n)$ if and only if $b\in \rm{BMO}_{\omega/\lambda}(\mathbb{R}^n)$(with identical $p,q,\alpha,\omega$ conditions).

For the endpoint case $p=1$, Harboure et al. \cite{HST} showed that $[b,I_\alpha]$ is bounded from $H^1(\mathbb{R}^n)$ to $L^{n/(n-\alpha)}(\mathbb{R}^n)$ precisely when $b$ is constant almost everywhere. That means, even if $b\in \mathrm{BMO}(\mathbb{R}^n)$, the commutator $[b,I_\alpha]$ may fail to be bounded from $H^1_\omega(\mathbb{R}^n)$ to $L^q_{\omega^q}(\mathbb{R}^n)$ ($q=n/(n-\alpha)$) for certain $\omega\in A_\infty$.
To characterize the boundedness properties of $[b,I_\alpha]$ on the weighted Hardy spaces,  the authors in \cite{HanWu} proved that there are nontrivial subspaces of $\rm {BMO}(\mathbb R^n)$, $\mathcal{BMO}_{\omega,p}(\mathbb R^n)$ (see Section 2.2 for the definition), when $b$ belongs to these subspaces, such that $[b,I_\alpha]$ is bounded from $H_{\omega^p}^p$ to $L_{\omega^q}^q$, or from $H_{\omega^p}^p$ to $H_{\omega^q}^q$ for certain $0<p\leq 1$ with $1/q=1/p-\alpha/n$ and $\omega\in A_\infty$.

Moreover, this limitation extends beyond fractional integrals. Harboure et al. \cite{HST} also demonstrated that the commutator $[b, T]$ formed by the Calder\'on--Zygmund singular integral operator $T$ and $b\in\mathrm{BMO}(\mathbb{R}^n)$, might fail to be bounded from $H^1_\omega(\mathbb{R}^n)$ to $L^1_\omega(\mathbb{R}^n)$ for certain $\omega\in A_\infty$. In an effort to better understand the behavior of $[b,T]$ on weighted Hardy spaces,  Huy et al. \cite{HK,Yang} demonstrated that for $b\in \rm BMO(\mathbb{R}^n)$, with $n/(n+1)<p\le 1$ and $\omega\in A_{p(n+1)/n}$, the commutators $\{[b,R_j]\}_{j=1}^n$ of the classical Riesz transforms are bounded from the weighted Hardy spaces $H^p_\omega(\mathbb{R}^n)$ to the weighted Lebesgue spaces $L^p_\omega(\mathbb{R}^n)$ if and only if $b\in\mathcal{BMO}_{\omega,p}(\mathbb{R}^n)$, where the ``only if" part of this statement is essentially dependent on the Riesz transform characterization of $H^p_\omega(\mathbb{R}^n)$ (see \cite[Theorem 1.5]{CCYY}).

To address the limitations in weighted Hardy spaces systematically,  Ky \cite{K2014} proposed a new space recently, called the Musielak--Orlicz Hardy spaces $H^{\varphi}(\mathbb{R}^n)$, which generalizes the weighted Hardy spaces and the Orlicz--Hardy spaces \cite{JY1, JY2, JY3}. Within this framework, Huy and Ky \cite{MOH} generalized the results in \cite{HK,Yang} from the  weighted Hardy spaces to the Musielak--Orlicz Hardy spaces. Precisely, they \cite{MOH} identified a suitable subspace $\mathcal{BMO}_\varphi$ (see Section 2.2 for the definition and properties) of $\rm{BMO}(\mathbb{R}^n)$. When $b$ is an element of $\mathcal{BMO}_\varphi$, the commutator $[b,T]$ demonstrates boundedness from the Musielak--Orlicz Hardy spaces $H^\varphi(\mathbb{R}^n)$ to the Musielak--Orlicz spaces $L^\varphi(\mathbb{R}^n)$. Conversely, based on the results of Riesz transform characterization of $H^\varphi(\mathbb{R}^n)$, if $b$ is in $\rm{BMO}(\mathbb{R}^n)$ and the commutator of classical Riesz transforms, $\{[b,R_j]\}_{j=1}^n$, are bounded from $H^\varphi(\mathbb R^n)$ to $L^\varphi(\mathbb R^n)$, then $b\in \mathcal{BMO}_\varphi(\mathbb R^n)$.

For the Riesz potential operators, Huy and Ky \cite{FRACTIONAL} proved the necessary and sufficient conditions for the boundedness of the fractional integral operator $I_\alpha$ from the Musielak--Orlicz Hardy spaces $H^{\varphi_1}(\mathbb{R}^n)$ to the Musielak--Orlicz spaces $L^{\varphi_2}(\mathbb{R}^n)$. Motivated by this result, a natural question arises: can the boundedness of Riesz potential commutators be used to characterize the space $\mathcal {BMO}_\varphi(\mathbb{R}^n)$? In this article, we investigate this open problem systematically and give an affirmative answer.

To facilitate the narration, we will introduce some definitions.
 A nonnegative function $\varphi$ on $\mathbb{R}^n \times[0, \infty)$ is called a Musielak--Orlicz function if for any $x \in \mathbb{R}^n$, the function $\varphi(x, \cdot)$ is an Orlicz function on $[0, \infty)$, and for any $t \in[0, \infty)$, the function $\varphi(\cdot, t)$ is measurable on $\mathbb{R}^n$. Notice that a function $\psi:[0, \infty) \rightarrow[0, \infty)$ is said to be an Orlicz function if it is nondecreasing, $\psi(0)=0$, $\psi(t)>0$ whenever $t \in(0, \infty)$ and $\lim _{t \rightarrow \infty} \psi(t)=\infty$.
	
	Let $\varphi$ be a Musielak--Orlicz function on $\mathbb{R}^n \times[0, \infty)$. $\varphi$ is said to be of uniformly lower (resp. upper) type $p$ with $p \in \mathbb{R}$, if there exists a positive constant $C$ such that for any $x \in \mathbb{R}^n,\;t \geq 0$ and $s \in(0,1]($ resp. $s \in[1, \infty))$,
	$\varphi(x, s t) \leq C s^p \varphi(x, t).$
	The critical uniformly lower type index and the critical uniformly upper type index of $\varphi$ are defined by
	$$
	i(\varphi):=\sup \{p \in \mathbb{R}: \varphi \text { is of uniformly lower type } p\},
	$$
	$$
	I(\varphi):=\sup \{p \in \mathbb{R}: \varphi \text { is of uniformly upper type } p\}.
	$$
	
	A locally integrable function $\varphi(\cdot, t): \mathbb{R}^n \rightarrow[0, \infty)$ is said to satisfy the uniformly Muckenhoupt condition $\mathbb{A}_q$, denoted by $\varphi \in \mathbb{A}_q$, if there exists a positive constant $C$ such that for any ball $B \subset \mathbb{R}^n$ and $t \in(0, \infty)$,
	when $q=1$,
	$$\frac{1}{|B|} \int_B \varphi(x, t) d x(\operatorname{ess} \sup _{x \in B}[\varphi(x, t)]^{-1}) \leq C,$$
	and when $q \in(1, \infty)$,	
	$$\frac{1}{|B|} \int_B \varphi(x, t) d x\left(\frac{1}{|B|} \int_B[\varphi(x, t)]^{-1 /(q-1)} d x\right)^{q-1} \leq C.$$
	We define $\mathbb{A}_{\infty}:=\bigcup_{q \in[1, \infty)} \mathbb{A}_q$. Also, we define the critical weight index of $\varphi \in \mathbb{A}_{\infty}$, that is,
	$q(\varphi):=\inf \left\{q \in[1, \infty): \varphi \in \mathbb{A}_q\right\}.$ We remark that the uniformly Muckenhoupt condition $\mathbb{A}_q$ was first defined by {\rm\cite{YLK}}.

The uniformly reverse  H\"older's condition is said to hold for a Musielak--Orlicz function $\varphi$ (denoted $\varphi\in \mathbb{RH}_r$), when there exists $r\in(1, \infty)$ such that for any ball $B$
$$
[\varphi]_{\mathbb{R H}_r}:=\sup _{t \in(0, \infty)} \sup _{B \subset \mathbb{R}^n}\left\{\left(\frac{1}{|B|} \int_B[\varphi(x, t)]^r d x\right)^{1 / r}\left(\frac{1}{|B|} \int_B \varphi(x, t) d x\right)^{-1}\right\}<\infty,
$$
and when $r=\infty$,
$$
[\varphi]_{\mathbb{R H}_{\infty}}:=\sup _{t \in(0, \infty)} \sup _{B \subset \mathbb{R}^n}\left\{\underset{x \in B}{\operatorname{ess} \sup } ~\varphi(x, t)\left(\frac{1}{|B|} \int_B \varphi(x, t) d x\right)^{-1}\right\}<\infty\;\;\;r=\infty.
$$
We define the critical weight index of $\varphi\in\mathbb{RH}_{\infty}$, that is $r(\varphi):=\sup\{r\in(1,\infty]:\varphi\in\mathbb{RH}_r\}$.
Clearly, when $\varphi(x,t)=\varphi(x)$ for all $t\in \mathbb R^n$, the uniformly Muckenhoupt condition  is the classical Muckenhoupt condition ${\rm A}_q$($q\in [1,\infty)$),
and the uniformly reverse
H\"older condition is the classical reverse H\"older condition ${\rm RH}_r$($r\in (1,\infty]$).

	A function $\varphi: \mathbb{R}^n \times[0, \infty) \rightarrow[0, \infty)$ qualifies as a growth function if the following conditions are satisfied: $\varphi$ is a Musielak--Orlicz function; $\varphi \in \mathbb{A}_{\infty}$; $\varphi$ is of uniformly lower type $p$ for some $p \in(0,1]$ and of uniformly upper type 1.

Below are the results of Riesz potential operators characterizing the Musielak--Orlicz Hardy spaces given by Huy and Ky \cite{FRACTIONAL}.

\medskip

\hspace{-14pt}{\bf Theorem A}(\hspace{0.5pt}\cite{FRACTIONAL})
\textit{ Let  $\alpha\in (0,n)$, $\varphi_1, \varphi_2$ be two growth functions.
Then, the necessary and sufficient condition for the boundedness of $I_\alpha$ from $H^{\varphi_1}(\mathbb{R}^n)$ into $L^{\varphi_2}(\mathbb{R}^n)$ is that there exists a positive constant $C$ such that, for all balls $B \subset \mathbb{R}^n$,
$|B|^{\frac{\alpha}{n}}\|\chi_B\|_{L^{\varphi_2}(\mathbb{R}^n)} \leq C\|\chi_B\|_{L^{\varphi_1}(\mathbb{R}^n)}$.}

\medskip

\hspace{-14pt}{\bf Theorem B}(\hspace{0.5pt}\cite{FRACTIONAL})
\textit{ Let $\alpha \in(0, n)$,  $\varphi_1, \varphi_2$ be two growth functions. Then, the necessary and sufficient condition for the boundedness of $I_\alpha$ from $H^{\varphi_1}(\mathbb{R}^n)$ into $H^{\varphi_2}(\mathbb{R}^n)$ is that there exists a positive constant $C$ such that, for all balls $B \subset \mathbb{R}^n$,
$|B|^{\frac{\alpha}{n}}\|\chi_B\|_{L^{\varphi_2}(\mathbb{R}^n)} \leq C\|\chi_B\|_{L^{\varphi_1}(\mathbb{R}^n)}$.}

\medskip

 {\bf Based on the above results, we characterize $\mathcal{BMO}_{\varphi}(\mathbb{R}^n)$ via the boundedness of the commutators of the Riesz potential operators on the Musielak--Orlicz Hardy spaces and obtain the following results.}

\begin{theorem}\label{T1.1}
    Let  $\alpha\in (0,1)$, $\varphi_1, \varphi_2$ be two growth functions, $\int_{\mathbb R^n}{\varphi_1 (x,(1+|x|)^{-n}})dx<\infty$,  and $(n-\alpha+1)i(\varphi_2)>nq(\varphi_2)$.
    There exists a positive constant $C_1$ such that, for all balls $B\subset\mathbb{R}^n$, $|B|^{\frac{\alpha}{n}}\|\chi_B\|_{L^{\varphi_2}(\mathbb{R}^n)}\le C_1\|\chi_B\|_{L^{\varphi_1}(\mathbb{R}^n)}$.
    If $b\in\mathrm{BMO}(\mathbb{R}^n)$, then the following two statements are equivalent:\\
    \rm{(i)} the commutator $[b,I_\alpha]$ is bounded from $H^{\varphi_1}(\mathbb{R}^n)$ to $L^{\varphi_2}(\mathbb{R}^n)$,\\
    \rm{(ii)} $b\in\mathcal{BMO}_{\varphi_1}(\mathbb{R}^n)$.
\end{theorem}

\begin{theorem}\label{T1.3}
Let $\alpha\in(0,1)$, $\varphi_1,\;\varphi_2$ be growth functions, $\int_{\mathbb R^n}{\varphi_1 (x,(1+|x|)^{-n}})dx<\infty$, $\int_{\mathbb R^n}{\varphi_2 (x,(1+|x|)^{-n}})dx<\infty$, $\varphi_2\in\mathbb{RH}_\infty$ and $(n-\alpha+1)i(\varphi_2)>nq(\varphi_2)$. There exists a positive constant $C_1$ such that, for all balls $B\subset\mathbb{R}^n$, $|B|^{\frac{\alpha}{n}}\|\chi_B\|_{L^{\varphi_2}(\mathbb{R}^n)}\le C_1\|\chi_B\|_{L^{\varphi_1}(\mathbb{R}^n)}$.
    If $b\in\mathrm{BMO}(\mathbb{R}^n)$, then the commutator $[b,I_\alpha]$ is bounded from $H^{\varphi_1}(\mathbb{R}^n)$ to $H^{\varphi_2}(\mathbb{R}^n)$.
\end{theorem}

Consider $\varphi_i(x,t)=\omega_i(x)t^{p_i}$,  with $p_i\in(0,1]$ and $\omega_i\in \rm A_\infty(\mathbb R^n)$, $i=1,2$. Direct calculation shows $\|\chi_B\|_{L^{\varphi_1}(\mathbb{R}^n)}=(\omega_1(B))^{1/p_1}$ and $\|\chi_B\|_{L^{\varphi_2}(\mathbb{R}^n)}=(\omega_2(B))^{1/p_2}$. Then inequality $|B|^{\frac{\alpha}{n}}\|\chi_B\|_{L^{\varphi_2}(\mathbb{R}^n)}\le C\|\chi_B\|_{L^{\varphi_1}(\mathbb{R}^n)}$ becomes $|B|^{\alpha/n}(\omega_2(B))^{1/p_2}\leq C(\omega_1(B))^{1/p_1}$. We immediately have the following results.

\begin{corollary}\label{Coro1}
 Let $\alpha\in (0,1)$, $0\leq\frac{1}{p_1}-\frac{1}{p_2}\leq\frac{\alpha}{n}$, $\omega_1,\omega_2\in \rm A_{\infty}$, $\int_{\mathbb R^n}\frac{\omega_1(x)}{(1+|x|)^{np_1}}dx<\infty$, $(n-\alpha+1)p_2>nq(\omega_2)$ (
	$q(\omega):=\inf \left\{q \in[1, \infty): \omega\in {\rm A}_q\right\}.$
).   There exists a positive constant $C_1$ such that, for all balls $B\subset\mathbb{R}^n$, $|B|^{\frac{\alpha}{n}}(\omega_2(B))^{1/p_2}\leq C_1(\omega_1(B))^{1/p_1}$.
 If $b\in\mathrm{BMO}(\mathbb{R}^n)$,  then the following two statements are equivalent:

    \rm{(i)} the commutator $[b,I_\alpha]$ is bounded from $H^{p_1}_{\omega_1}(\mathbb{R}^n)$ to $L^{p_2}_{\omega_2}(\mathbb{R}^n)$,

    \rm{(ii)} $b\in\mathcal{BMO}_{\omega_1,p_1}(\mathbb{R}^n)$(see Section 2.2 for the definition).
\end{corollary}

\begin{corollary}\label{Coro2}
 Let $\alpha\in (0,1)$, $0\leq\frac{1}{p_1}-\frac{1}{p_2}\leq\frac{\alpha}{n}$, $\omega_1\in \rm A_{\infty}$, $\omega_2\in  RH_\infty$, $\int_{\mathbb R^n}\frac{\omega_1(x)}{(1+|x|)^{np_1}}dx<\infty$, $\int_{\mathbb R^n}\frac{\omega_2(x)}{(1+|x|)^{np_2}}dx<\infty$, and $(n-\alpha+1)p_2>nq(\omega_2)$.
  There exists a positive constant $C_1$ such that, for all balls $B\subset\mathbb{R}^n$, $|B|^{\frac{\alpha}{n}}(\omega_2(B))^{1/p_2}\leq C_1(\omega_1(B))^{1/p_1}$. If $b\in\mathcal{BMO}_{\omega_1,p_1}(\mathbb{R}^n)\cap\mathcal{BMO}_{\omega_2,p_2,u}(\mathbb{R}^n)$ $(1<u<\infty)$, then the commutator $[b,I_\alpha]$ is bounded from $H^{p_1}_{\omega_1}(\mathbb{R}^n)$ to $H^{p_2}_{\omega_2}(\mathbb{R}^n)$.
\end{corollary}

\begin{remark}
It should be noted that Corollaries {\rm\ref{Coro1}} and {\rm\ref{Coro2}} improve the results of {\rm\cite{HanWu}} in the following ways: (i) Here, we consider the necessary and sufficient conditions for the boundedness of $[b,I_\alpha]$ on the weighted Hardy spaces, whereas {\rm\cite{HanWu}} only examines the sufficient conditions; (ii) The weight function of the space can be different Muckenhoupt weight functions $\omega_1$, $\omega_2$, while {\rm\cite{HanWu}} requires the weight function to be uniform.
\end{remark}

The rest of this paper is organized as follows. In Section \rm\ref{S2} we will recall some related definitions and  auxiliary lemmas. The characterization of $\mathcal{BMO}_\varphi$ will be given in Section \rm\ref{S3}. Finally, we will
present the corresponding results for the commutators of fractional integrals with general
homogeneous kernels in Section \rm\ref{S4}.
We remark that some ideas in our arguments are taken from \rm\cite{HanWu,Yang}, in which the Riesz potential operators,  the commutators of classical Calder\'on--Zygmund operators were dealt with on the weighted Hardy spaces.
	
	Finally, we make some conventions on notation. Throughout this paper, we denote by $C$ a positive constant which is independent of the main parameters, but it may vary from line to line. We denote   $E^c=\mathbb{R}^n\backslash E$ is the complementary set of any measurable subset $E$ of $\mathbb{R}^n$.

\section{Preliminaries}\label{S2}
In this section, we recall some auxiliary facts and lemmas, which will be used in our arguments.

\subsection{Musielak--Orlicz Hardy spaces}
The Musielak--Orlicz space $L^{\varphi}(\mathbb{R}^n)$ comprises all measurable functions $f$ on $\mathbb{R}^n$ satisfying $\int_{\mathbb{R}^n} \varphi\Big(x,\frac{|f(x)|}{\lambda}  \Big) d x<\infty$ for some $\lambda>0$, carrying the (quasi-)norm
$$
\|f\|_{L^{\varphi}(\mathbb{R}^n)}:=\inf \{\lambda>0: \int_{\mathbb{R}^n} \varphi\Big(x,\frac{|f(x)|}{\lambda}\Big) d x \leq 1\}.
$$
For any $A\subset{\mathbb{R}^n}$, we define
$$\|f\|_{L^{\varphi}(A)}:=\inf \{\lambda>0: \int_A \varphi\Big(x,\frac{|f(x)|}{\lambda}\Big) d x \leq 1\}.$$

\hspace{-14pt}{\bf Definition of Musielak--Orlicz Hardy spaces.} The  Musielak--Orlicz Hardy space $H^{\varphi}(\mathbb{R}^n)$ is the space of all $f\in \mathcal{S}'(\mathbb{R}^n) $, $\phi\in \mathcal S$ satisfying $\mathcal{M}_\phi f \in L^{\varphi}(\mathbb{R}^n)$  endowed with the (quasi-)norm
$$
\|f\|_{H^{\varphi}(\mathbb{R}^n)}:=\|\mathcal{M}_\phi f\|_{L^{\varphi}(\mathbb{R}^n)},
$$
where the maximal function of a distribution $f$ for each $x\in\mathbb{R}^n$ and $\phi_t(\cdot)=t^{-n}\phi(t^{-1}\cdot)$ is given by
$$\mathcal{M}_\phi f(x):=\sup\limits_{t\in(0,\infty)}|f*\phi_t(x)|.$$

To characterize $H^{\varphi}(\mathbb{R}^n)$ via atomic decomposition, we need both  the definition of $L_\varphi^q(\mathbb{R}^n)$ spaces, and the precise notion of atoms in this framework. We present  the definition of  spaces first.
For any ball $B$ in $\mathbb{R}^n$, we denote $L_{\varphi}^q(B), q\in(q(\varphi),\infty]$ as the collection of all measurable functions $f$ on $\mathbb{R}^n$ satisfy $\supp f \subset B$ such that
$$
\|f\|_{L_{\varphi}^q(B)}:= \begin{cases}\sup\limits_{t>0}(\frac{1}{\varphi(B, t)}{\int_B|f(x)|^q \varphi(x, t) d x})^{1 / q}<\infty, & 1 \leq q<\infty, \\ \|f\|_{L^{\infty}(\mathbb R^n)}<\infty, & q=\infty.\end{cases}
$$

\hspace{-14pt}{\bf Definition of $(H^\varphi,q)$-atom.}
 Let $\varphi$ be a growth function, and let $q \in(q(\varphi), \infty]$. A measurable function $a$ is said to be an $(H^{\varphi}, q)$-atom related to the ball $B \subset \mathbb{R}^n$ if\\
${\rm(i)}\;\operatorname{supp} a \subset B$,\\
${\rm(ii)}\;\|a\|_{L_{\varphi}^q(B)} \leq\|\chi_B\|_{L^{\varphi}(\mathbb{R}^n)}^{-1}$,\\
${\rm(iii)}\;\int_{\mathbb{R}^n} a(x) d x=0$.

\begin{lemma}\label{2.2}\rm(\hspace{0.5pt}\cite{YLK})
Let $\varphi$ be a growth function, $\mathcal{X}$ be is a quasi-Banach space. Suppose that one of the following holds:

 {\rm(i)}\ $q\in (q(\varphi),\infty)$ and $T$ is a sublinear operator defined on the space of all finite linear combinations of $(H^\varphi,q)-atoms$ with the property that
$$\sup \{  \| Ta \|  _{\mathcal{X} } :\text{a is an } (H^\varphi,q)-atom \} \le C.$$

{\rm(ii)}\ $T$ is a sublinear operator defined on the space of all finite linear combinations of continuous  $(H^\varphi,\infty)-atoms$ with the property that
$$\sup \{  \| Ta \|  _{\mathcal{X} } :\text{a is a continuous } (H^\varphi,\infty)-atom \} \le C.$$
Then, $T$ admits a unique continuous extension to a bounded sublinear operator from  $H^{\varphi } (\mathbb{R}^n)$ into $\mathcal{X}$.
\end{lemma}

\begin{lemma}\label{2.3}\rm(\hspace{0.5pt}\cite{MOH})
    Let $q \in[1, \infty), r \in(1, \infty)$, and let $\varphi \in {\mathbb{A}}_q(\mathbb{R}^n) \cap \mathbb{RH}_r(\mathbb{R}^n)$ be of uniformly lower type $p$. Then,
$$
\Big(\frac{|E|}{|B|}\Big)^{q / p} \le C\frac{\|\chi_E\|_{L^{\varphi}(\mathbb{R}^n)}}{\|\chi_B\|_{L^{\varphi}(\mathbb{R}^n)}} \le C\Big(\frac{|E|}{|B|}\Big)^{1-1 / r}
$$
for any ball $B \subset \mathbb{R}^n$ and measurable subset $E \subset B$.
\end{lemma}

\begin{lemma}\label{2.4}\rm(\hspace{0.5pt}\cite{MOH})
If $\varphi \in \mathbb{A}_{\infty} (\mathbb{R}^n)$, then there exist $q\in[1,\infty)$ and $r\in(1,\infty)$ such that $\varphi \in \mathbb{A}_q(\mathbb{R}^n)\cap \mathbb{RH}_{r}(\mathbb{R}^n)$, moreover,\\
$$\Big(\frac{|E|}{|B|}\Big)^{q} \le C\frac{\varphi(E, t)}{\varphi(B, t)} \le C\Big(\frac{|E|}{|B|}\Big)^{1-1 / r}$$\\
for all $t\in[0,\infty)$, balls $B\subset \mathbb{R}^n$, and measurable subset $E\subset B$.
\end{lemma}

\begin{lemma}\label{2.1}\rm(\hspace{0.5pt}\cite{MOH})
     Let $\varphi$ be a growth function and $q \in(q(\varphi), \infty]$. Then,\\
$$\int_B \varphi(x,|f(x)|) d x \leq C \varphi(B,\|f\|_{L_{\varphi}^q(B)})$$
for all ball $B \subset \mathbb{R}^n, f \in L_{\varphi}^q(B)$.
\end{lemma}

\subsection{$\mathcal{BMO}_{\varphi,u}(\mathbb R^n)$ space}
Given $\varphi\in \mathbb A_\infty$ be a growth function with $\int_{\mathbb{R}^n}\varphi(x,(1+|x|)^{-n})dx<\infty$ and  $u\in(1,\infty]$, a locally integrable function $b\in L_{loc}^1(\mathbb{R}^n)$ is said to be in $\mathcal{BMO}_{\varphi,u}(\mathbb{R}^n)$ if
$$\|b\|_{\mathcal{BMO}_{\varphi, u}}:=\sup _{B \subset \mathbb{R}^n}\Big\{\frac{|B|^{1 / u}}{\|\chi_B\|_{L^{\varphi}(\mathbb{R}^n)}}(\int_B|b(x)-b_B|^{u'} d x)^{1 / u'}\||\cdot-x_B|^{-n}\|_{L^{\varphi}(B^c)}\Big\}<\infty,$$
where $b_B:=\frac{1}{|B|}\int_B|b(x)|dx$ and the supremum is taken over all balls $B:=B(x_B,r_B)\subset\mathbb{R}^n$.

A locally integrable function $b$ is said to be in $\mathrm{BMO}(\mathbb{R}^n)$ if
$$\|b\|_{\mathrm{BMO}(\mathbb{R}^n)}:=\sup\limits_{B\subset{\mathbb{R}^n}}
\frac{1}{|B|}\int_B|b(x)-b_B|dx<\infty.$$

By the definition of ${\mathcal{BMO}_{\varphi,u}(\mathbb{R}^n)}$, we have the following basic facts.

\medskip

{\bf Basic facts.}

$(\mathrm{i})$ For $\varphi\in \mathbb A_\infty, u=\infty$, we denote $\mathcal{BMO}_\varphi(\mathbb{R}^n)$ by $\mathcal{BMO}_{\varphi,\infty}(\mathbb{R}^n)$, which were introduced in \cite{MOH}. Consider $\varphi(x,t)=\omega(x)t^{p}$ where $p\in(0,1]$, $\omega\in \rm A_\infty(\mathbb R^n)$, we denote $\mathcal{BMO}_{\varphi,u}(\mathbb{R}^n)$ by $\mathcal{BMO}_{\omega,p,u}(\mathbb{R}^n)$ which were introduced in \cite{HanWu} .

$(\mathrm{ii})$ Let $\varphi$ be a growth function with $\int_{\mathbb R^n}{\varphi (x,(1+|x|)^{-n}})dx<\infty$. If $1<u_1<u_2<\infty$, then
 \begin{align*}
\mathcal{BMO}_{\varphi,u_1}(\mathbb R^n)\subset\mathcal{BMO}_{\varphi,u_2}(\mathbb R^n)
\subset\mathcal{BMO}_{\varphi}(\mathbb R^n)\subset \mathrm{BMO}(\mathbb R^n).
 \end{align*}

\begin{lemma}\label{2.6}\rm(\hspace{0.5pt}\cite{MOH})
Let $\varphi$ be a growth function. Then, for all balls $B\subset\mathbb{R}^n$, and $b\in\mathrm{BMO}(\mathbb{R}^n)$,
$$
    \int_B\varphi(x,|b(x)-b_B|)dx\le C\varphi(B,\|b\|_{\mathrm{BMO}(\mathbb{R}^n)}).
$$
\end{lemma}

\begin{lemma}\label{2.7}\rm(\hspace{0.5pt}\cite{MOH})
Let $\varphi$ be a growth function with $\int_{\mathbb R^n}{\varphi (x,(1+|x|)^{-n}})dx<\infty$. Then, for all $b\in {\mathcal{B M O}_{\varphi}(\mathbb{R}^n)}$ and all $(H^\varphi,\infty)$-atom $a$ related to the ball $B:=B(x_B,r_B)$, we have
\begin{align*}
\|(b-b_B)a\|_{H^\varphi(\mathbb{R}^n)}\le C\|b\|_{\mathcal{B M O}_\varphi(\mathbb{R}^n)}.
\end{align*}
\end{lemma}


\begin{lemma}\label{2.9}\rm(\hspace{0.5pt}\cite{Yang})
    Let $f$ be a measurable function such that $\supp f\subset B:=B(x_B,r_B)$ with some $x_B\in\mathbb{R}^n$ and $r_B\in(0,\infty)$. Then there exists a postive constant $C$, depending only on $\phi$ and $n$, such that, for all $x\notin B$,
    \begin{align*}
        \frac{1}{|x-x_0|^n}\Big|\int_{B(x_B,r_B)}f(y)dy\Big|\le C \mathcal{M}_{\phi}f(x).
    \end{align*}
\end{lemma}

\section{Proof of Main Results}\label{S3}
This section is devoted to the proofs of our main theorems. Before proving Theorem \ref{T1.1}, we first establish the following lemma.

\begin{lemma}\label{3.1}
Let $\alpha\in(0,1)$, $\varphi_1,\;\varphi_2$ be as in Theorem \ref{T1.1}.  Then there exists a positive constant $C$, such that for any $b\in\mathrm{BMO(\mathbb{R}^n)}$ and any continuous $(H^{\varphi_1},\infty)$-atom $a$ related to the ball $B:=B(x_B,r_B)\subset\mathbb{R}^n$,
\begin{align*}
\|(b-b_B)I_\alpha a\|_{L^{\varphi_2}(\mathbb{R}^n)}\le C\|b\|_{\mathrm{BMO}(\mathbb{R}^n)}.
\end{align*}
\end{lemma}
\begin{proof}
 Due to $\alpha\in(0,1)$, for every $x\in E_1:=2B$, we know that
\begin{align*}
    |I_\alpha a(x)|&\le\int_B\frac{|a(y)|}{|x-y|^{n-\alpha}}dy\\
&\le\|a\|_{L^\infty(\mathbb{R}^n)}\int_B\frac{1}{|x-y|^{n-\alpha}}dy\\
    &\le C|B|^{\frac{\alpha}{n}}\|\chi_B\|_{L^{\varphi_1}(\mathbb{R}^n)}^{-1}\le C\|\chi_B\|_{L^{\varphi_2}(\mathbb{R}^n)}^{-1}.
\end{align*}

For all $x\in E_j:=2^jB\setminus2^{j-1}B,\;j\ge 2$, by using the mean value theorem and the
cancellation conditions of atom $a$, we can write
\begin{align*}
    |I_\alpha a(x)|
    &=\Big|\int_B\Big(\frac{1}{|x-y|^{n-\alpha}}-\frac{1}{|x-x_0|^{n-\alpha}}\Big)a(y)dy\Big|\\
    &\le C\int_B\frac{|y-x_0|}{|x-x_0|^{n-\alpha+1}}|a(y)|dy\\
    &\le C2^{-j(n-\alpha+1)}|B|^\frac{\alpha}{n}\|\chi_B\|_{L^{\varphi_1}(\mathbb{R}^n)}^{-1}\\
    &\le C2^{-j(n-\alpha+1)}\|\chi_B\|_{L^{\varphi_2}(\mathbb{R}^n)}^{-1}.
\end{align*}
Next, we can take $p\in(1,i(\varphi_2))$ together with $q\in(q(\varphi_2),\infty)$ such that $(n-\alpha+1)p>nq$.
Thus, for all $\lambda>0$, by Lemmas \ref{2.4}, \ref{2.6}, we get
\begin{align*}
    \int_{E_j}\varphi_2\Big(x,\frac{|b(x)-b_B|}{\lambda}\Big)dx&\le C\int_{E_j}\varphi_2\Big(x,\frac{|b(x)-b_{2^jB}|+|b_{2^jB}-b_B|}{\lambda}\Big)\\
    &\le C\varphi_2\Big(2^jB,\frac{\|b\|_{\mathrm{BMO}(\mathbb{R}^n)}}{\lambda}\Big) + Cj\varphi_2\Big(2^jB,\frac{\|b\|_{\mathrm{BMO}(\mathbb{R}^n)}}{\lambda}\Big)\\
    &\le C2^{jnq}(j+1)\varphi_2\Big(B,\frac{\|b\|_{\mathrm{BMO}(\mathbb{R}^n)}}{\lambda}\Big).
\end{align*}
Further,  integrating the above estimates for $|I_\alpha a(x)|$ with the uniformly lower type $p$ property of $\varphi_2$ and Lemma \ref{2.4}, we infer that
\begin{align*}
\int_{E_j}&\varphi_2\Big(x,\frac{|(b(x)-b_B)I_\alpha a(x)|}{\lambda}\Big)dx\\
    &\le C2^{-j(n-\alpha+1)p}\int_{E_j}\varphi_2\Big(x,\frac{|b(x)-b_B|\|\chi_B\|_{L^{\varphi_2}(\mathbb{R}^n)}^{-1}}{\lambda}\Big)dx\\
    &\le C(j+1)2^{-j[(n-\alpha+1)p-nq]}\varphi_2\Big(B,\frac{\|b\|_{\mathrm{BMO}(\mathbb{R}^n)}\|\chi_B\|_{L^{\varphi_2}(\mathbb{R}^n)}^{-1}}{\lambda}\Big)dx,
\end{align*}
for all $j\ge1$.
Together with $(n-\alpha+1)p>nq$, we have
\begin{align*}
\|(b-&b_B)I_\alpha a\|_{L^{\varphi_2}(\mathbb{R}^n)}\\
&\le C\inf\{\lambda>0:\sum_{j=1}^{\infty} \int_{E_j}\varphi_2\Big(x,\frac{|(b(x)-b_B)I_\alpha a(x)|}{\lambda}\Big)dx\le1\}\\
&\le C\inf\{\lambda>0:\sum_{j=1}^{\infty}(j+1)2^{-j[(n-\alpha+1)p-nq]}\varphi_2\Big(B,\frac{\|b\|_{\mathrm{BMO}(\mathbb{R}^n)}\|\chi_B\|_{L^{\varphi_2}(\mathbb{R}^n)}^{-1}}{\lambda}\Big)\le1\}\\
&\le C\|b\|_{\mathrm{BMO}(\mathbb{R}^n)}.
\end{align*}
Hence, we complete the proof of Lemma \ref{3.1}.
\end{proof}

Now, we are in the position to prove Theorem \ref{T1.1}.
\begin{proof}[\bf Proof of Theorem \ref{T1.1}]
    We begin with proving that (ii) implies (i). By Lemma \ref{2.2}, it suffices to prove that, for any continuous $(H^{\varphi_1},\infty)$-atom $a$ related to the ball $B:=B(x_B,r_B)\subset\mathbb{R}^n$ with $x_B\in\mathbb{R}^n$ and $r_B\in(0,\infty)$,
    \begin{align*}
        \|[b,I_\alpha](a)\|_{L^{\varphi_2}(\mathbb{R}^n)}\le C\|b\|_{\mathcal{BMO}_{\varphi_1}(\mathbb{R}^n)}.
    \end{align*}
    Note that
    \begin{align*}
        [b,I_\alpha](a)(x)=[b(x)-b_B]I_\alpha a(x)-I_\alpha((b-b_B)a)(x).
    \end{align*}
    Indeed, follows from Lemmas \ref{2.7}, \ref{3.1} and Theorem A, we have
    \begin{align*}
        \|[b,I_\alpha](a)\|_{L^{\varphi_2}(\mathbb{R}^n)}&\le C(\|(b-b_B)I_\alpha a\|_{L^{\varphi_2}(\mathbb{R}^n)}+\|I_\alpha((b-b_B)a)\|_{L^{\varphi_2}(\mathbb{R}^n)})\\
        &\le C(\|b\|_{\mathrm{BMO}(\mathbb{R}^n)}+\|(b-b_B)a\|_{H^{\varphi_1}(\mathbb{R}^n)})
        \le C\|b\|_{\mathcal{BMO}_{\varphi_1}(\mathbb{R}^n)}.
    \end{align*}

    We next prove that (i) implies (ii). For every ball $B=B(x_B,r_B)\subset\mathbb{R}^n$, defining atomic function
    \begin{align*}
        \tilde a:=\frac{1}{2\|\chi_B\|_{L^{\varphi_1}(\mathbb{R}^n)}}(f-f_B)\chi_B,
    \end{align*}
where $f:=\mathrm{sign}(b-b_B)$. We can verify that $\tilde a$ is an $(H^{\varphi_1},\infty)$-atom related to the ball $B$. By Lemma \ref{2.9}, for all $x\notin B$, we have
\begin{align*}
    |x-x_B|^{-n}\frac{1}{2\|\chi_B\|_{L^{\varphi_1}(\mathbb{R}^n)}}\int_B|b(x)-b_B|dx&=|x-x_B|^{-n}\int_B[b(x)-b_B]\tilde a(x)dx\\
    &\le C\mathcal{M}_{\phi}([b-b_B]\tilde a)(x).
\end{align*}
This yields that
\begin{align*}
    \frac{1}{\|\chi_B\|_{L^{\varphi_1}(\mathbb{R}^n)}}\int_B|b(x)-b_B|dx\||x-x_B|^{-n}\|_{L^{\varphi_1}(B^c)}
    &\le C\|\mathcal{M}_\phi([b-b_B]\tilde a)\|_{L^{\varphi_1}(\mathbb{R}^n)}.
\end{align*}
We can further estimate the $\|\mathcal{M}_\phi([b-b_B]\tilde a)\|_{L^{\varphi_1}(\mathbb{R}^n)}$.
Note that
$$[b,I_\alpha](\tilde a)(x)=[b(x)-b_B]I_\alpha \tilde a(x)-I_\alpha((b-b_B)\tilde a)(x).$$
Indeed, follows from Theorem \ref{T1.1}, Lemma \ref{3.1}, we have
\begin{align*}
    \|I_\alpha([b-b_B]\tilde a)\|_{L^{\varphi_2}(\mathbb{R}^n)}&\le C(\|[b,I_\alpha](\tilde a)\|_{L^{\varphi_2}(\mathbb{R}^n)}+\|(b-b_B)I_\alpha \tilde a\|_{L^{\varphi_2}(\mathbb{R}^n)})\\
    &\le C(\|[b,I_\alpha]\|_{{H^{\varphi_1}(\mathbb{R}^n)}\to{L^{\varphi_2}(\mathbb{R}^n)}}+\|b\|_{\mathrm{BMO}(\mathbb{R}^n)}).
\end{align*}
Thus, by Theorem A,
 we obtain
\begin{align*}
    \|b\|_{\mathcal{BMO}_{\varphi_1}}\le C(\|[b,I_\alpha]\|_{{H^{\varphi_1}(\mathbb{R}^n)}\to{L^{\varphi_2}(\mathbb{R}^n)}}+\|b\|_{\mathrm{BMO}(\mathbb{R}^n)}).
\end{align*}
This shows that $b\in\mathcal{BMO}_{\varphi_1}(\mathbb{R}^n)$, and completing the proof of Theorem \ref{T1.1}.
\end{proof}
\medskip

Before proving the Theorem \ref{T1.3}, we begin with giving the following auxiliary lemma.
\begin{lemma}\label{3.2}
Let  $\alpha\in(0,1)$, $\varphi_1,\;\varphi_2$ be growth functions with $\varphi_2\in\mathbb{RH}_\infty$ and $(n-\alpha+1)i(\varphi_2)>nq(\varphi_2)$. There exists a positive constant $C_1$ such that $|B|^{\frac{\alpha}{n}}\|\chi_B\|_{L^{\varphi_2}(\mathbb{R}^n)}\le C_1\|\chi_B\|_{L^{\varphi_1}(\mathbb{R}^n)}$ holds for all balls $B\subset\mathbb{R}^n$. Let $a$ be a continuous $(H^{\varphi_1},\infty)$-atom with $\supp a\subset B(x_B,r_B)$, then $I_\alpha a$ can be decomposed into the composition of $(H^{\varphi_2},q)$-atoms $\{a_j\}_{j=1}^\infty$ related to the balls $\{2^{j}B\}_{j=1}^\infty$,  $q\in(q(\varphi_2),\infty)$. Precisely,$$I_\alpha a(x)=\sum_{j=1}^\infty\lambda_ja_j(x),$$ where exists a positive constant $C$ and $\varepsilon>0$ such that $|\lambda_j|\le C2^{-j\varepsilon}$ for all $j\in\mathbb{Z}^+$.
\end{lemma}
\begin{proof}
    For convenience, we denote $M(x)=I_\alpha a(x)$. Let $E_0=B_1=B(x_B,2r_B),\;E_j:=B_{j+1}\setminus B_j\;(j=1,2,\dots)$, let $M_j$ be the mean value of $M(x)$ on $E_j$.
    Define
$\alpha_j(x)=(M(x)-M_j)\chi_{E_j}(x).$
It is clear that $\supp\alpha_j\subset B_{j+1}$, and $\int_{B_{j+1}}\alpha_j(x)dx=0, ~j=0,1,2,\ldots$. Choose $q\in(q(\varphi_2),\infty)$, then by H\"older's inequality and $\varphi _2\in\mathbb{A}_q$,
\begin{align*}
    &\Big(\frac{1}{\varphi_2(B_{j+1},t)}\int_{B_{j+1}}|(M_j)\chi_{E_j}(x)|^{q}\varphi_2(x,t)dx\Big)^{1/q}\\
&\leq C\big(\varphi_2(E_j,t)\big)^{1/q}\big(\varphi_2(B_{j+1},t)\big)^{-1/q}
                |E_j|^{-1}\\
&\qquad\qquad\times\Big(\int_{B_{j+1}}|M(x)\chi_{E_j}(x)|^{q}
                \varphi_2(x,t)dx\Big)^{1/q}\Big(\int_{B_{j+1}}
                \varphi_2(x,t)^{-1/(q-1)}dx\Big)^{1-1/q}\\
&\leq C\big(\varphi_2(B_{j+1},t)\big)^{1/q}\big(\varphi_2(B_{j+1},t)\big)^{-1/q}
                |E_j|^{-1}\\
&\qquad\qquad\times\Big(\int_{B_{j+1}}|M(x)\chi_{E_j}(x)|^{q}
                \varphi_2(x,t)dx\Big)^{1/q}\Big(\int_{B_{j+1}}
                \varphi_2(x,t)dx\Big)^{-1/q}|B_{j+1}|\\
&\leq C\| M\chi_{E_j}\|_{L_{\varphi_2}^{q}(B_{j+1})}
                \frac{|B_{j+1}|}{|E_j|}
\leq C\| M\chi_{E_j}\|_{L_{\varphi_2}^{q}(B_{j+1})}.
\end{align*}
 Thus, we have
 \begin{align*}
\|\alpha_j\|_{L_{\varphi_2}^{q}(B_{j+1})}
&\leq
           \| M\chi _{E_j}\|_{L_{\varphi_2}^{q}(B_{j+1})}
        +\| (M_j)\chi_{E_j}\|_{L_{\varphi_2}^{q}(B_{j+1})}\\
&\leq C\| M\chi _{E_j}\|_{L_{\varphi_2}^{q}(B_{j+1})}.
\end{align*}

For $j=0$, we choose $1<p_0<n/\alpha,\;q_0>q$ such that $\alpha/n+1/q_0=1/p_0$. From the way $p_0$ and $q_0$ are chosen, we can derive that $I_\alpha$ is bounded from $L^{p_0}(\mathbb{R}^n)$ into $L^{q_0}(\mathbb{R}^n)$. Due to $\varphi\in\mathbb{RH}_\infty\subset\mathbb{RH}_{q_0\setminus(q_0-q)}$, by H\"older's inequality and Lemma \ref{2.3}, we yield that
\begin{align*}
\|M\chi_{E_0}\|_{L_{\varphi_2}^q(B_1)}
&\leq\varphi_2(B_1,t)^{-1/q}\Big(\int_{B_1}
                    |I_\alpha a(x)|^{q_0}dx\Big)^{1/q_0}\Big(\int_{B_1}
                    \varphi_2(x,t)^{q_0/(q_0-q)}dx\Big)^{(q_0-q)/q_0q}\\
&\leq|B_1|^{-1/q_0}\|I_\alpha a\|_{L^{q_0}(B_1)}\\
&\qquad\times\Big[\Big(\frac{1}{|B_1|}\int_{B_1}\varphi_2(x,t)dx\Big)^{-1}\Big(\frac{1}{|B_1|}\int_{B_1}\varphi_2(x,t)^{q_0/(q_0-q)}dx\Big)^{(q_0-q)/q_0}\Big]^{1/q}\\
&\le C|B_1|^{-1/q_0}\|a\|_{L^{p_0}(\mathbb{R}^n)}\\
&\le C|B_1|^{-1/q_0}|B|^{1/p_0}\|\chi_B\|_{L^{\varphi_1}(\mathbb{R}^n)}^{-1}\le C\|\chi_{B_1}\|_{L^{\varphi_2}(\mathbb{R}^n)}^{-1}.
\end{align*}
Thus, we have $ \|\alpha_0\|_{L_{\varphi_2}^{q}(B_1)}\leq C\|\chi_{B_1}\|_{L^{\varphi_2}(\mathbb{R}^n)}^{-1}$.

For $j=1,2,\ldots$, we have
$$\| \alpha_j\|_{L_{\varphi_2}^{q}(B_{j+1})}
\leq C2^{(j+1)(nq/p-n+\alpha-1)}\|\chi_{B_{j+1}}\|_{L^{\varphi_2}(\mathbb{   R}^n)}^{-1}.
$$
 The proof of the above inequality is standard. Precisely, we select $p\in (0, i(\varphi_2))$ and $q\in (q(\varphi_2,\infty))$ such that $nq/p-n+\alpha-1<0$, because $(n-\alpha+1)i(\varphi_2)>nq(\varphi_2)$. Then, by the mean value Theorem and the cancellation conditions of $a$ and Lemma \ref{2.3},
 \begin{align*}
     &\Big(\frac{1}{\varphi_2(B_{j+1},t)}\int_{E_j}
       |I_\alpha a(x)|^{q}\varphi_2(x,t)dx\Big)^{1/q}\\
&=\Big(\frac{1}{\varphi_2(B_{j+1},t)}\int_{E_j}
       \Big|\int_Ba(y)\Big[\frac{1}{|x-y|^{n-\alpha}}-\frac{1}{|x-x_0|^{n-\alpha}}\Big]dy\Big|^{q}\varphi_2(x,t)dx\Big)^{1/q}\\
 &\leq C\Big(\frac{1}{\varphi_2(B_{j+1},t)}\int_{E_j}
       \Big|\int_Ba(y)\frac{|y-x_0|}{|x-x_0|^{n-\alpha+1}}dy\Big|^{q}\varphi_2(x,t)dx\Big)^{1/q}\\
&\leq C\Big(\frac{1}{\varphi_2(B_{j+1},t)}\int_{E_j}
         \frac{r_B^{q}}{(2^jr_B)^{q(n-\alpha+1)}}
       \Big|\int_Ba(y)dy\Big|^{q}\varphi_2(x,t)dx\Big)^{1/q}\\
&\leq C\Big(\frac{1}{\varphi_2(B_{j+1},t)}\int_{E_j}
         \frac{r_B^{q(n+1)}}{(2^jr_B)^{q(n-\alpha+1)}}
       \varphi_2(x,t)dx\Big)^{1/q}\|\chi_B\|_{L^{\varphi_2}(\mathbb{R}^n)}^{-1}|B|^{-\alpha/n}\\
&\leq C\frac{r_B^{(n+1)}}{(2^jr_B)^{(n-\alpha+1)}}
        \|\chi_{B_{j+1}}\|_{L^{\varphi_2}(\mathbb{R}^n)}^{-1}
        \frac{|B_{j+1}|^{q/ p}}{|B|^{q/ p}}|B|^{-\alpha/n}\\
 &\leq C2^{(j+1)(nq/p-n+\alpha-1)}
          \|\chi_{B_{j+1}}\|_{L^{\varphi_2}(\mathbb{R}^n)}^{-1}         .
 \end{align*}

 Moreover,
$$M(x)-\sum_{j=0}^{\infty} \alpha_j(x)=\sum_{j=0}^{\infty}(M_j) \chi_{E_j}(x).$$
Denote $b_j:=\int_{E_j}M(x)dx$, we write $\sum_{j=0}^\infty(M_j)\chi_{E_j}=\sum_{j=0}^\infty b_j\frac{(M_j)\chi_{E_j}}{b_j}$. By the fact $\sum_{k=0}^\infty b_k=0$, we get
\begin{align*}
\sum_{j=0}^\infty b_j\frac{(M_j)\chi_{E_j}(x)}{b_j}
=&\sum_{j=0}^\infty\Big(\sum_{k=j}^\infty b_k-\sum_{k=j+1}^\infty b_k\Big)\frac{(M_j)\chi_{E_j}(x)}{b_j}
=:\sum_{j=0}^\infty \beta_j (x).
\end{align*}
It is easy to see that $\supp \beta_j\subset B_{j+2},$ and $\int_{B_{j+2}}\beta_j(x)dx=0, j=0,\,1,\,2,\ldots$.
\begin{align*}
    \|\beta_j\|_{L_{\varphi_2}^{q}(B_{j+2})}
\leq \Big\|\sum_{k=j+1}^\infty b_k\frac{(M_{j+1})\chi_{E_{j+1}}}{b_{j+1}}\Big\|_{L_{\varphi_2}^{q}(B_{j+2})}
    +\Big\|\sum_{k=j+1}^\infty b_k\frac{(M_j)\chi_{E_j}}{b_j}\Big\|_{L_{\varphi_2}^{q}(B_{j+1})}.
\end{align*}
For the first part,  by H\"older's inequality, Lemma \ref{2.3} and $\varphi_2 \in \mathbb {RH}_\infty\cap\mathbb{A}_q$, we have
\begin{align*}
&\frac{1}{\varphi_2(B_{j+2},t)}\int_{B_{j+2}}\Big|\sum_{k=j+1}^\infty
           b_k\frac{M_{j+1}\chi_{E_{j+1}}(x)}{b_{j+1}}\Big|^{q}\varphi_2(x,t)dx\\
\leq &\sum_{k=j+1}^\infty\frac{1}{\varphi_2(B_{j+2},t)}\int_{B_{j+2}}
        \Big|\frac{b_k\chi_{E_{j+1}}(x)}{|E_{j+1}|}\Big|^{q}\varphi_2(x,t)dx\\
\leq &C\sum_{k=j+1}^\infty\frac{\varphi_2(E_{j+1},t)}{\varphi_2(B_{j+2},t)}
        \frac{1}{|E_{j+1}|^{q}}\|M\chi_{E_k}\|_{L_{\varphi_2}^{q}(B_{k+1})}^{q}\\
 &\qquad\qquad\times       \varphi_2(B_{k+1},t)\Big(\int_{B_{k+1}} \varphi_2(x,t)^{-1/(q-1)}dx\Big)^{q-1}\\
\leq &C\sum_{k=j+1}^\infty
        \frac{1}{|E_{j+1}|^{q}}|B_{k+1}|^{q}\|M\chi_{E_k}\|_{L_{\varphi_2}^{q}(B_{k+1})}^{q}
        \varphi_2(B_{k+1},t)\big(\varphi_2(B_{k+1},t)\big)^{-1}\\
\leq &C\sum_{k=j+1}^\infty
        \frac{|B_{k+1}|^{q}}{|E_{j+1}|^{q}}2^{(k+1)q(nq/ p-n+\alpha-1)}
        \|\chi_{B_{k+1}}\|_{L^{\varphi_2}(\mathbb{R}^n)}^{-q}\\
\leq &C\sum_{k=j+1}^\infty
        \frac{|B_{k+1}|^{q}}{|E_{j+1}|^{q}}2^{(k+1)q(nq/ p-n+\alpha-1)}
        \|\chi_{B_{j+2}}\|_{L^{\varphi_2}(\mathbb{R}^n)}^{-q}\frac{|B_{j+2}|^{q}}{|B_{k+1}|^{q}}\\
\leq &C2^{(j+2)q(nq/ p-n+\alpha-1)}
        \|\chi_{B_{j+2}}\|_{L^{\varphi_2}(\mathbb{R}^n)}^{-q}.
\end{align*}
Thus,
\begin{align*}
    \Big\|\sum_{k=j+1}^\infty b_k\frac{(M_{j+1})\chi_{E_{j+1}}}{b_{j+1}}\Big\|_{L_{\varphi_2}^{q}(B_{j+2})}\leq C\|\chi_{B_{j+2}}\|^{-1}_{{L^{\varphi_2}}(\mathbb{R}^n)}2^{(j+2)(nq/ p-n+\alpha-1)}.
\end{align*}
 Similarly,
\begin{align*}
    \Big\|\sum_{k=j+1}^\infty b_k\frac{(M_j)\chi_{E_j}}{b_j}\Big\|_{L_{\varphi_2}^{q}(B_{j+1})}\leq C\|\chi_{B_{j+1}}\|^{-1}_{L^{\varphi_2}(\mathbb{R}^n)}2^{(j+1)(nq/ p-n+\alpha-1)}.
\end{align*}
Combining the estimates above, we have
\begin{align*}
\|\beta_j\|_{L_{\varphi_2}^{q}(B_{j+2})}\leq C\|\chi_{B_{j+2}}\|^{-1}_{L^{\varphi_2}(\mathbb{R}^n)}2^{(j+2)(nq/ p-n+\alpha-1)}.
\end{align*}
Thus,
\begin{align*}
M(x)&=\sum_{j=0}^\infty 2^{-j\varepsilon}\big(2^{j\varepsilon}\alpha_j(x)\big)+\sum_{j=0}^\infty 2^{-j\varepsilon}\big(2^{j\varepsilon}\beta_j (x)\big)\\
&=\sum_{j=0}^\infty 2^{-j\varepsilon}\Big(2^{j\varepsilon}\alpha_{j-1}(x)+2^{j\varepsilon}\beta_{j-2} (x)\Big)=:\sum_{j=1}^\infty \lambda_j a_j(x),
\end{align*}
 where $\alpha_{-1}=\beta_{-1}=\beta_{-2}=0$, $\varepsilon=n-\alpha-nq/p+1>0$. It is easy to see that $\supp{a_j}\subset B_j$,
 $\|a_j\|_{L_{\varphi_2}^{q}(B_{j})}
 \leq C\|\chi_{B_{j}}\|_{L^{\varphi_2}(\mathbb{R}^n)}^{-1},$
 $\int_{\mathbb{R}^n}a_j(x)dx=0$ and $\sum_{j=1}^\infty |\lambda_j|^q<\infty$. Therefore, these functions $\{a_j\}$ defined above are all $(H^{\varphi_2},q)$-atoms related to the balls $\{B_j\}_{j\in \mathbb{Z_+}}$.
\end{proof}

We now have a proof of Theorem \ref{T1.3}.
\begin{proof}[\bf Proof of Theorem \ref{T1.3}]
As the proof of Theorem \ref{T1.1}, by Lemma \ref{2.2}, it suffices to prove that for any continuous $(H^{\varphi_1},\infty)$-atom $a$ related to the ball $B=B(x_B,r_B)\subset\mathbb{R}^n$,
\begin{align*}
    \|[b,I_\alpha](a)\|_{H^{\varphi_2}(\mathbb{R}^n)}\le C\big(\|b\|_{\mathcal{BMO}_{\varphi_1}(\mathbb{R}^n)}+\|b\|_{\mathcal{BMO}_{\varphi_2,u}(\mathbb{R}^n)}\big).
\end{align*}
By Theorem B and Lemma \ref{2.7}, we obtain
\begin{align*}
    \|I_\alpha((b-b_B)a)\|_{H^{\varphi_2}(\mathbb{R}^n)}\le C\|(b-b_B)a\|_{H^{\varphi_1}(\mathbb{R}^n)}\le C\|b\|_{\mathcal{BMO}_{\varphi_1}(\mathbb{R}^n)}.
\end{align*}
By Lemma \ref{3.2}, we know that $I_\alpha a$ can be decomposed into $(H^{\varphi_2},q)$-atoms for $q\in (q(\varphi_2),\infty)$ as follows:
\begin{align*}
    I_\alpha a(x)=\sum_{j=1}^{\infty} \lambda_j a_j(x),
\end{align*}
where $\{a_j\}_{j=1}^\infty$ are $(H^{\varphi_2},s)$-atoms $(uq(\varphi_2)<s<\infty)$ related to the balls $\{2^{j}B\}_{j=1}^\infty$, and there exists a positive constant $C$ and $\varepsilon>0$ such that $|\lambda_j|\le C2^{-j\varepsilon}$ for all $j\in{\mathbb{Z}}$.
Thus, we can simplify the proof by showing that for every $(H^{\varphi_2},s)$-atom $a_j$,
\begin{align*}
    \|(b-b_{2^{j+1}B})a_j\|_{H^{\varphi_2}(\mathbb{R}^n)}\le C\|b\|_{\mathcal{BMO}_{\varphi_2,u}(\mathbb{R}^n)}, \;\;\;\forall j\in \mathbb Z_+.
\end{align*}
We write
\begin{align*}
\|\mathcal{M} _\phi((b-b_{2^{j+1}B})a_j)\|_{L^{\varphi_2}(\mathbb{R}^n)}
& \le \|\mathcal{M} _\phi((b-b_{2^{j+1}B})a_j)\|_{L^{\varphi_2}(2^{j+1}B)}\\
    &\qquad+\|\mathcal{M} _\phi((b-b_{2^{j+1}B})a_j)\|_{L^{\varphi_2}((2^{j+1}B)^c)}\\
&=:I_1+I_2.
\end{align*}
For $I_1$, choose $q_1\in(q(\varphi_2),s)$ satisfies $1/q_1=1/s+1/l$, by the H\"older's inequality, Lemma \ref{2.3}, for all $t\in[0,\infty)$, we have,
\begin{align*}
\Big[\frac{1}{\varphi_2(2^{j+1}B,t)}&\int_{2^{j+2}B}|\mathcal{M}_\phi
         ((b-b_{2^{j+1}B})a_j)(x)|^{q_1}\varphi_2(x,t)dx\Big]^{1/q_1}\\
&\le C\Big[\frac{1}{\varphi_2(2^{j+1}B,t)}\int_{\mathbb{R}^n}
         |(b(x)-b_{2^{j+1}B})a_j|^{q_1}\varphi_2(x,t)dx\Big]^{1/q_1}\\
&\le C\Big[\frac{1}{\varphi_2(2^{j+1}B,t)}\int_{2^{j+1}B}
         |b(x)-b_{2^{j+1}B}|^{l}\varphi_2(x,t)dx\Big]^{1/l}\\
  &\qquad\qquad\times
         \Big[\frac{1}{\varphi_2(2^{j}B,t)}\int_{2^{j}B}|a_j(x)|^{s}
          \varphi_2(x,t)dx\Big]^{1/s}\\
&\le C\|b\|_{\rm BMO(\mathbb{R}^n)}
         \|a_j\|_{L_{\varphi_2}^s(2^{j}B)}\\
&\le C\|b\|_{\rm BMO(\mathbb{R}^n)}\|\chi_{2^{j}B}\|_{L^{\varphi_2}(\mathbb{R}^n)}^{-1}.
\end{align*}
Hence, by Lemma \ref{2.1}, we have
\begin{align*}
\|\mathcal{M} _\phi(&(b-b_{2^{j+1}B})a_j)\|_{L^{\varphi_2}(2^{j+1}B)}\\
&:=\inf\{\lambda>0:\int_{2^{j+1}B}\varphi_2(x,\frac{|\mathcal{M}_\phi((b-b_{2^{j+1}B})a_j)(x)|}{\lambda})dx\le1\}\\
&\le C\inf\{\lambda>0:\varphi_2(2^{j+1}B,\frac{\|\mathcal{M}_\phi((b-b_{2^{j+1}B})a_j)\|_{L_{\varphi_2}^{q_1}(2^{j+1}B)}}{\lambda})\le1\}\\
&\le C\|b\|_{\mathcal{BMO}_{\varphi_2,u}(\mathbb{R}^n)}
\end{align*}
Now we estimate $I_2$. For every $x\in(2^{j+1}B)^c$ and $y\in2^{j}B$, we see that $|x-y|\approx|x-x_0|$. And by H\"older's inequality for $s/u\;(1<u<s<\infty)$ , $\varphi_2\in\mathbb{A}_{s/u}$, for the maximal function estimate, we have
\begin{align*}
&\mathcal{M} _\phi((b-b_{2^{j+1}B})a_j)(x)\\
&:=\sup_{t>0}\frac{1}{t^n}\int_{2^{j}B}
|b(y)-b_{2^{j+1}B}||a_j(y)|\big|\phi(\frac{x-y}{t})\big|dy\\
&\le C\frac{1}{|x-x_0|^n}\int_{2^{j}B}
     |b(y)-b_{2^{j+1}B}||a_j(y)|dy\\
&\le C\frac{1}{|x-x_0|^n}(
       \int_{2^{j}B}|b(y)-b_{2^{j+1}B}|^{u'}dy)^{1/u'}
            (\int_{2^{j}B}|a_j(y)|^{u}dy)^{1/u}\\
&\le C\frac{1}{|x-x_0|^n}(
       \int_{2^{j}B}|b(y)-b_{2^{j+1}B}|^{u'}dy)^{1/u'}\\
&\qquad\qquad\times (\int_{2^{j}B}|a_j(y)|^{s}\varphi_2(y,t)dy)^{1/s}
             (\int_{2^{j}B}\varphi_2(y,t)^{-u/(s-u)}dy)^{1/u-1/s}\\
&\le C\frac{1}{|x-x_0|^n}(
       \int_{2^{j+1}B}|b(y)-b_{2^{j+1}B}|^{u'}dy)^{1/u'}
            \|a_j\|_{L_{\varphi_2}^s(\mathbb{R}^n)}|2^{j+1}B|^{1/u}\\
&\le C\frac{|2^{j+1}B|^{1/u}}{|x-x_0|^n\|\chi_{2^{j+1}B}\|_{L^{\varphi_2}(\mathbb{R}^n)}}(
       \int_{2^{j+1}B}|b(y)-b_{2^{j+1}B}|^{u'}dy)^{1/u'}.
\end{align*}
This leads to the boundedness for $I_2$,
\begin{align*}
I_2\leq &C\frac{|2^{j+1}B|^{1/u}}{\|\chi_{2^{j+1}B}\|_{L^{\varphi_2}(\mathbb{R}^n)}}
            \||\cdot-x_0|^{-n}\|_{L^{\varphi_2}((2^{j+1}B)^c)}\\
&\qquad\times(\int_{2^{j+1}B}|b(y)-b_{2^{j+1}B}|^{u'}
        dy)^{1/u'}\leq C\|b\|_{\mathcal {BMO}_{\varphi_2,u}(\mathbb{R}^n)}.
       \end{align*}
Combining with  the estimate for $I_2$, we obtain
\begin{align*}
\|[b(x)-b_B]I_\alpha a\|_{H^{\varphi_2}(\mathbb{R}^n)}
&=\Big\|\sum_{j=1}^\infty[(b-b_{2^{j+1}B})+(b_{2^{j+1}B}-b_B)]
               \lambda_ja_j\Big\|_{H^{\varphi_2}(\mathbb{R}^n)}\\
&\le C\sum_{j=1}^\infty|\lambda_j|(\|(b-b_{2^{j+1}B})a_j\|_{H^{\varphi_2}(\mathbb{R}^n)}
             +\|(b_B-b_{2^{j+1}B})a_j\|_{H^{\varphi_2}(\mathbb{R}^n)})\\
&\le C\Big(\|b\|_{\mathcal{BMO}_{\varphi_2,u}(\mathbb{R}^n)}
           \sum_{j=1}^\infty2^{-j\eps}
    +\|b\|_{{\rm BMO}(\mathbb{R}^n)}
    \sum_{j=1}^\infty(j+1)2^{-j\eps}\Big)\\
&\le C\|b\|_{\mathcal{BMO}_{\varphi_2,u}(\mathbb{R}^n)}.
\end{align*}
The commutator estimate follows that
\begin{align*}
    \|[b,I_\alpha](a)\|_{H^{\varphi_2}(\mathbb{R}^n)}\le C\big(\|b\|_{\mathcal{BMO}_{\varphi_1}(\mathbb{R}^n)}+\|b\|_{\mathcal{BMO}_{\varphi_2,u}(\mathbb{R}^n)}\big).
\end{align*}
This completes the proof of Theorem \ref{T1.3}.
\end{proof}

\section{Further Results}\label{S4}
In this section, we present the corresponding results for the commutators of fractional integrals with homogeneous kernels.
Suppose that $\Omega\in L^1(\mathbb{S}^{n-1})$ is a homogeneous of degree zero on $\mathbb{R}^n$, and let $\mathbb{S}^{n-1}$ denote the unit sphere in $\mathbb{R}^n\;(n\ge2)$ equipped with normalized Lebesgue measure $d\sigma=d\sigma(x')$ and $x':=x/|x|$ for any $x\ne0$.
For $0<\alpha<n$ and $b\in L_{\mathrm{loc}}(\mathbb{R}^n)$, we consider the general fractional integrals with homogeneous kernels

 $$T_{\Omega,\alpha}f(x)=\int_{\mathbb{R}^n}\frac{\Omega(x-y)}{|x-y|^{n-\alpha}}f(y)dy$$
 and the corresponding commutator
 $$[b,T_{\Omega,\alpha}]f(x)=\int_{\mathbb{R}^n}[b(x)-b(y)]\frac{\Omega(x-y)}{|x-y|^{n-\alpha}}f(y)dy.$$
 It can be easily concluded that $T_{\Omega,\alpha}=I_\alpha$ and $[b,T_{\Omega,\alpha}]=[b,I_\alpha]$ when $\Omega=1$.

 The boundedness of $[b,T_{\Omega, \alpha}]$ on Lebesgue and weighted Lebesgue spaces have been studied extensively. For rough kernel operators,  Ding and Lu \cite{DL2} established the $(L^p_{\omega^p},L^q_{\lambda^q})$-boundedness under the conditions $1<p<n/\alpha$, $1/q=1/p-\alpha/n$, $\omega, \lambda\in A_{p,q}$ and $b\in BMO_{\omega/\lambda}(\mathbb{R}^n)$. This result was extended to fractional integrals with log-Dini type kernels by Segovia and Torrea \cite{ST}. Notably, when $\Omega\in C(S^{n-1})$, Guo et al. \cite{GLW} established a complete characterization for continuous kernels, showing that the $[b, T_{\Omega,\alpha}]$ is $(L_{\omega^p}^p,L_{\lambda^q}^q)$-bounded if and only if $b\in BMO_{\omega/\lambda}(\mathbb{R}^n)$. In \cite{DLL2003}, Ding et al. investigated the boundedness properties of the operator $T_{\Omega,\alpha}$ within the context of weighted Hardy spaces. Further, we \cite{HanWu} considered the boundedness of $[b,T_{\Omega,\alpha}]$ on the weighted Hardy spaces, $b\in \mathcal{BMO}_{\omega,p}$.

 Inspired by Theorems A and B, it is natural to consider the behaviors of $ T_{\Omega, \alpha}$ in the Musielak--Orlicz  Hardy spaces. The similar arguments with slight
modifications to those in proving Theorems A and  B in \cite{FRACTIONAL}, we can obtain the following theorems.

\begin{theorem}\label{T4.1}
    Let $\beta\in(0,1]$, $\alpha\in(0,\beta)$, $\Omega\in Lip_\beta(\mathbb{S}^{n-1})$, $\varphi_1, \varphi_2$ be two growth functions,  and $(n-\alpha+\beta)i(\varphi_2)>nq(\varphi_2)$, there exists a positive constant $C_1$ such that, for all balls $B\subset\mathbb{R}^n$, $|B|^{\frac{\alpha}{n}}\|\chi_B\|_{L^{\varphi_2}(\mathbb{R}^n)}\le C_1\|\chi_B\|_{L^{\varphi_1}(\mathbb{R}^n)}$. Then $T_{\Omega,\alpha}$ is bounded from $H^{\varphi_1}(\mathbb{R}^n)$ to $L^{\varphi_2}(\mathbb{R}^n)$.
\end{theorem}

\begin{theorem}\label{T4.3}
      Let $\beta\in(0,1]$, $\alpha\in(0,\beta)$, $\Omega\in Lip_\beta(\mathbb{S}^{n-1})$, $\varphi_1,\varphi_2$ be growth functions with $\varphi_2\in\mathbb{RH}_\infty$ and $(n-\alpha+\beta)i(\varphi_2)>nq(\varphi_2)$, there exists a positive constant $C_1$ such that, for all balls $B\subset\mathbb{R}^n$, $|B|^{\frac{\alpha}{n}}\|\chi_B\|_{L^{\varphi_2}(\mathbb{R}^n)}\le C_1\|\chi_B\|_{L^{\varphi_1}(\mathbb{R}^n)}$, then  $T_{\Omega,\alpha}$ is bounded from $H^{\varphi_1}(\mathbb{R}^n)$ to $H^{\varphi_2}(\mathbb{R}^n)$.
\end{theorem}

Further, by employing a similar approach to Theorems \ref{T1.1} and \ref{T1.3}, we can derive the boundedness of the commutators $[b,T_{\Omega,\alpha}]$ in the Musielak--Orlicz Hardy spaces.

\begin{theorem}\label{T4.2}
     Let $\beta\in(0,1]$, $\alpha\in(0,\beta)$, $\Omega\in Lip_\beta(\mathbb{S}^{n-1})$, $\varphi_1, \varphi_2$ be two growth functions,  and $\varphi_2\in\mathbb{A}_\infty$ be a Musielak--Orlicz function satisfies $1\leq i(\varphi_2)\leq n/(n-\alpha)$ and $(n-\alpha+\beta)i(\varphi_2)>nq(\varphi_2)$, there exists a positive constant $C_1$ such that, for all balls $B\subset\mathbb{R}^n$, $|B|^{\frac{\alpha}{n}}\|\chi_B\|_{L^{\varphi_2}(\mathbb{R}^n)}\le C_1\|\chi_B\|_{L^{\varphi_1}(\mathbb{R}^n)}$. If $b\in\mathcal{BMO}_{\varphi_1}(\mathbb{R}^n)$, then the commutator $[b,T_{\Omega,\alpha}]$ is bounded from $H^{\varphi_1}(\mathbb{R}^n)$ to $L^{\varphi_2}(\mathbb{R}^n)$.
\end{theorem}

\begin{theorem}\label{T4.4}
      Let $\beta\in(0,1]$, $\alpha\in(0,\beta)$, $\Omega\in Lip_\beta(\mathbb{S}^{n-1})$, $\varphi_1,\varphi_2$ be growth functions with $\int_{\mathbb{R}^n}\varphi_1(x,(1+|x|^{-n})dx<\infty$, $\int_{\mathbb{R}^n}\varphi_2(x,(1+|x|^{-n})dx<\infty$, and $\varphi_2\in \mathbb{RH}_\infty$,  and $(n-\alpha+\beta)i(\varphi_2)>n$, there exists a positive constant $C_1$ such that, for all balls $B\subset\mathbb{R}^n$, $|B|^{\frac{\alpha}{n}}\|\chi_B\|_{L^{\varphi_2}(\mathbb{R}^n)}\le C_1\|\chi_B\|_{L^{\varphi_1}(\mathbb{R}^n)}$. If $b\in\mathcal{BMO}_{\varphi_1}(\mathbb{R}^n)\cap\mathcal{BMO}_{\varphi_2,u}(\mathbb{R}^n)$ $(1<u<\infty)$, then the commutator $[b,T_{\Omega,\alpha}]$ is bounded from $H^{\varphi_1}(\mathbb{R}^n)$ to $H^{\varphi_2}(\mathbb{R}^n)$.
\end{theorem}

\begin{remark}
The above results generalize the results in {\rm\cite{DLL2003,HanWu}}, which the boundedness of  fractional integral operators  $T_{\Omega,\alpha}$ and their commutators $[b,T_{\Omega,\alpha}]$ from weighted Hardy spaces to more general Musielak--Orlicz  Hardy spaces.
\end{remark}

\noindent{\bf Acknowledgments:}
Yanyan Han is supported by the National Natural Science Foundation of China(No. 12526513), the Fundamental Research Funds for the Central Universities, and the Research Funds of People's Public Security University of China(PPSUC) (No. 2024JKF02ZK08), and the Open Research Fund of Hubei Key Laboratory of Mathematical Sciences (Central China Normal University), Wuhan 430079, P. R. China. Hongwei Huang is supported by the Natural Science Foundation of Fujian Province of China (No. 2023J01026). Huoxiong Wu is supported by the National Natural Science Foundation of China (No. 12271041).

\medskip

\noindent {\bf Compliance with Ethical Standard}

\medskip

\noindent {Conflict of Interests:} The authors declare that there is no conflict of interests regarding the publication of this paper.

\medskip

\noindent {Data Availability Statements:} No data-sets were generated or analysed during the current study.

\end{document}